\documentclass[12pt]{amsart}
\usepackage{amssymb,amsmath,amsthm}
\def\bold{\bf}

\def\ab{{\bold a}}

\def\tb{{\bold t}}
\def\xb{{\bold x}}

\def\0b{{\bold 0}}

\usepackage{bm}
\bmdefine{\Bzero}{0}
\bmdefine{\Bone}{1}
\def\Bone{{\bf 1}}

\def\tb{{\bf t}}

\def\RR{{\mathbb R}}
\def\ZZ{{\mathbb Z}}

\newtheorem{Theorem}{Theorem}[section]
\newtheorem{Lemma}[Theorem]{Lemma}
\newtheorem{Corollary}[Theorem]{Corollary}
\newtheorem{Proposition}[Theorem]{Proposition}
\newtheorem{Remark}[Theorem]{Remark}

\newtheorem{Example}[Theorem]{Example}

\begin{document}
\title{Toric rings and ideals of nested configurations}
\author{Hidefumi Ohsugi and Takayuki Hibi}
\subjclass[2000]{Primary:13P10, Secondary:52B20}
\begin{abstract}
The toric ring together with the toric ideal arising from a nested configuration is studied,
with particular attention given to the algebraic study of normality of the toric ring as well
as the Gr\"obner bases of the toric ideal.
One of the combinatorial applications of these
algebraic findings leads to insights on smooth $3 \times 3$ transportation polytopes.
\end{abstract}

\maketitle

\section*{Introduction}
Toric rings and toric ideals
play a central role
in combinatorial and computational aspects of
commutative algebra.
In \cite{AHOT2},
from a viewpoint of algebraic statistics,
the concept of nested configurations
was introduced.  In the present paper,  
the toric ring together with the toric ideal
arising from a nested configuration
will be studied in detail.

Let $K[\tb]=K[t_1, \ldots, t_d]$ denote 
the polynomial ring
in $d$ variables 
over a field $K$.
A {\em (point) configuration} of $K[\tb]$
is a finite set $A =\{ \tb^{{\bf a}_1},\ldots,  \tb^{{\bf a}_n} \}$
 of monomials belonging to $K[\tb]$
satisfying that
there exists a vector ${\bf w} \in \RR^d$ 
such that
${\bf w} \cdot {\bf a}_i = 1$ for all $1 \leq i \leq n$.
We will associate each configuration $A$ of $K[\tb]$
with the homogeneous semigroup ring
$K[A]$, called the {\em toric ring} of $A$, which is
the subalgebra of $K[\tb]$ generated by
the monomials belonging to $A$.
The toric ring $K[A]$ is called {\em normal}
if $K[A]$ is integrally closed in its field of fractions.
It is known that $K[A]$ is normal if and only if 
$
\ZZ_{\geq 0} \{ {\bf a}_1 ,\ldots, {\bf a}_n \}
=
\ZZ \{ {\bf a}_1 ,\ldots, {\bf a}_n \}
\cap 
{\mathbb Q}_{\geq 0} \{ {\bf a}_1 ,\ldots, {\bf a}_n \}
$.
See, e.g., \cite[Proposition 13.5]{Stu}.
In addition, $K[A]$ is called {\em very ample}
if 
$$
(\ZZ \{ {\bf a}_1 ,\ldots, {\bf a}_n \}
\cap 
{\mathbb Q}_{\geq 0} \{ {\bf a}_1 ,\ldots, {\bf a}_n \})
\setminus
\ZZ_{\geq 0} \{ {\bf a}_1 ,\ldots, {\bf a}_n \}
$$
is a finite set.
In particular, $K[A]$ is very ample if $K[A]$ is normal.

Let $K[\xb] = K[x_1,\ldots,x_n]$ denote
the polynomial ring over $K$ in $n$ variables 
with each $\deg(x_i) = 1$.
The {\em toric ideal} $I_A$ of $A$ 
is the kernel of the surjective homomorphism 
$\pi \, : \, K[\xb] \to K[A]$
defined by setting
$\pi(x_i) = \tb^{\ab_i}$ for each $1 \leq i \leq n $.
It is known (e.g., \cite[Section 4]{Stu})
that the toric ideal
$I_A$ is generated by those homogeneous
binomials $u - v$, where
$u$ and $v$ are monomials of $K[\xb]$, with
$\pi(u) = \pi(v)$.
Fix a monomial order $<$ on $K[\xb]$.
The {\em initial monomial} 
${\rm in}_<(f)$ of
$0 \neq f \in I_A$ with respect to $<$
is the biggest monomial appearing in $f$ with respect to $<$.
The {\em initial ideal} of $I_A$ with respect to $<$ is the ideal ${\rm in}_<(I_A)$
of $K[\xb]$ generated by all initial monomials ${\rm in}_<(f)$ 
with $0 \neq f \in I_A$.
An initial ideal ${\rm in}_<(I_A)$ is called {\em quadratic} (resp. {\em squarefree})
if ${\rm in}_<(I_A)$
is generated by quadratic (resp. squarefree) monomials.
Let, in general, ${\mathcal G}$ be a finite subset of
$I_A$ and write ${\rm in}_<({\mathcal G})$ for the ideal
$\left< {\rm in}_<(g) \ | \  g \in {\mathcal G} \right>$ of $K[X]$.
A finite set ${\mathcal G}$ of $I_A$ is said to be a
{\em Gr\"obner basis} of $I_A$ with respect to $<$ if 
${\rm in}_<({\mathcal G}) = {\rm in}_<(I_A)$.
It is known that a Gr\"obner basis of $I_A$ with respect to
$<$ always exists. 
Moreover, if ${\mathcal G}$ is a Gr\"obner basis of $I_A$, then 
$I_A$ is generated by
${\mathcal G}$.
A Gr\"obner basis ${\mathcal G}$ of $I_A$ is called {\em quadratic}
if ${\rm in}_<({\mathcal G})$ is quadratic.
We are interested in two implications below:
\begin{center}
$I_A$ has a squarefree initial ideal
$\Longrightarrow $
$K[A]$ is normal
$\Longrightarrow $
$K[A]$ is very ample;
\end{center}
\begin{center}
$I_A$ has a quadratic Gr\"obner basis
$\Longrightarrow $
$K[A]$ is Koszul
$\Longrightarrow $
$I_A$ is generated by quadratic binomials.
\end{center}
It is known that each of the converse of them is false
in general.
See, e.g., \cite{noneofwhose, graphquadratic}.

For the sake of simplicity, let $A =\{ \tb^{{\bf a}_1},\ldots,  \tb^{{\bf a}_n} \}$ 
be a configuration of
$K[{\bf t}]$ with the following properties:
\begin{itemize}
\item
$|{\bf a}_j| = r$ for each $1 \leq j \leq n$;
\item
$t_i$ divides the monomial $\tb^{{\bf a}_1} \cdots \tb^{{\bf a}_n}$
for each $1 \leq i \leq d$.
\end{itemize}
(Note that any configuration is isomorphic to such a configuration.)
Assume that, for each $1 \leq i  \leq d$, 
a configuration 
$B_i=\{m_1^{(i)},\ldots,m_{\lambda_i}^{(i)}\}$ 
of a polynomial ring 
$K[{\bf u}^{(i)}] = K[ u_1^{(i)},\ldots,u_{\mu_i}^{(i)}]$
in $\mu_i$ variables over $K$ is given.
Then the {\it nested configuration} \cite{AHOT2} 
arising from $A$ and $B_1,...,B_d$
is
the configuration
$$
A(B_1,\ldots,B_d) := 
\left\{ \left.
{m}_{j_1}^{(i_1)} \cdots {m}_{j_r}^{(i_r)} 
\ \right| \ 
t_{i_1} \cdots t_{i_r} \in A, \ \ 
1 \leq j_k \leq \lambda_{i_k} 
\mbox{ for } 1 \leq k \leq r
\right\}
$$
of the polynomial ring
$K[{\bf u}^{(1)},\ldots,{\bf u}^{(d)}]$
in $\sum_{i=1}^{d} \mu_i$ variables over $K$.
Here, $t_{i_1} \cdots t_{i_r} \in A$ is not necessarily squarefree.
If $A = \{t_1 t_2\}$, then $K[A(B_1, B_2)]$ is the Segre product of
$K[B_1]$ and $K[B_2]$.
Moreover, if $A=\{t_1^m\}$, then $K[A(B_1)]$ is the $m$-th Veronese subring
of $K[B_1]$.

\begin{Example}
{\em
Let $A=\{t_1^2,t_1 t_2\}$, $B_1= \{u_1^2,u_1u_2,u_2^2\}$ and
$B_2=\{v_1^2 v_2, v_1v_2^2\}$.
Then, the nested configuration $A(B_1,B_2)$ consists of the monomials
$$
u_1^4,
u_1^3u_2,
u_1^2u_2^2,
u_1u_2^3,
u_2^4, \ \ 
u_1^2   v_1^2 v_2,
u_1u_2  v_1^2 v_2,
u_2^2   v_1^2 v_2,
u_1^2   v_1v_2^2,
u_1u_2  v_1v_2^2,
u_2^2   v_1v_2^2.
$$
Then, the matrices
$$
M_A = 
\left(
\begin{array}{cc}
2 & 1\\
0 & 1
\end{array}
\right),
M_{B_1} = 
\left(
\begin{array}{ccc}
2 & 1 & 0\\
0 & 1 & 2
\end{array}
\right),
M_{B_2} = 
\left(
\begin{array}{cc}
2 & 1\\
1 & 2
\end{array}
\right),
$$
$$
M_{A(B_1,B_2)} = 
\left(
\begin{array}{ccccc|cccccc}
4 & 3 & 2 & 1 & 0 & 2 & 1 & 0 & 2 & 1 & 0\\
0 & 1 & 2 & 3 & 4 & 0 & 1 & 2 & 0 & 1 & 2\\
\hline
0 & 0 & 0 & 0 & 0 & 2 & 2 & 2 & 1 & 1 & 1\\
0 & 0 & 0 & 0 & 0 & 1 & 1 & 1 & 2 & 2 & 2
\end{array}
\right)
$$
correspond to the configurations $A$, $B_1$, $B_2$ and $A(B_1,B_2)$, respectively.
}
\end{Example}
One of the fundamental facts of the nested configuration is

\begin{Theorem}[\cite{AHOT2}]
\label{old}
If each of the toric ideals $I_A$, $I_{B_1}, \ldots , I_{B_d}$ 
possesses a quadratic
Gr\"obner basis,
then
the toric ideal $I_{A(B_1,\ldots,B_d) }$
possesses a quadratic Gr\"obner basis.
\end{Theorem}

In Section 1, we study the normality of the toric ring arising
from a nested configuration.
Our first main result is Theorem \ref{normality}:
if each of
$K[A], K[B_1], \ldots, K[B_d]$
are normal then
$K[A(B_1, \ldots, B_d)]$ is also normal.
In general -- see Example \ref{ex13} -- the converse does
not hold.
However,
Corollary \ref{introduction}
guarantees
that, when $A$ consists of squarefree monomials,
each of $K[A], K[B_1], \ldots, K[B_d]$
is normal if and only if
$K[A(B_1, \ldots, B_d)]$ is normal.

In Section $2$,
we study Gr\"obner bases
of the toric ideal
arising from a nested configuration.
A natural generalization of Theorem \ref{old} 
will be obtained.  In fact, 
Theorem \ref{pcase} together with
Theorem \ref{maincase}
guarantees that
if each of
$I_A, I_{B_1}, \ldots, I_{B_d}$
possesses a Gr\"obner basis consisting of
binomials of degree at most $p$, then
$I_{A(B_1, \ldots, B_d)}$
possesses a Gr\"obner basis consisting of
binomials of degree at most $\max ( 2,p )$.
Moreover, 
if each of
$I_A, I_{B_1}, \ldots, I_{B_d}$
possesses a squarefree initial ideal,
then
$I_{A(B_1, \ldots, B_d)}$
possesses a squarefree initial ideal.

In Section $3$,
as one of the combinatorial applications of
our algebraic theory of nested configurations,
we discuss the toric ideal of a multiple
of the Birkhoff polytope ${\mathcal B}_3$.
Here ${\mathcal B}_3$ is the convex hull of
$$
\sigma_1=
\left(
\begin{array}{ccc}
1 & 0 & 0\\
0 & 1 & 0\\
0 & 0 & 1
\end{array}
\right),
\sigma_2=
\left(
\begin{array}{ccc}
0 & 1 & 0\\
0 & 0 & 1\\
1 & 0 & 0
\end{array}
\right),
\sigma_3=
\left(
\begin{array}{ccc}
0 & 0 & 1\\
1 & 0 & 0\\
0 & 1 & 0
\end{array}
\right),
$$
$$
\sigma_4=
\left(
\begin{array}{ccc}
1 & 0 & 0\\
0 & 0 & 1\\
0 & 1 & 0
\end{array}
\right),
\sigma_5=
\left(
\begin{array}{ccc}
0 & 1 & 0\\
1 & 0 & 0\\
0 & 0 & 1
\end{array}
\right),
\sigma_6=
\left(
\begin{array}{ccc}
0 & 0 & 1\\
0 & 1 & 0\\
1 & 0 & 0
\end{array}
\right)
$$
in $\RR^{3 \times 3}$.
The toric ideal of ${\mathcal B}_3$ is the toric ideal
of the configuration
$$
B_1
=
\{
u_{11} u_{22} u_{33},
u_{12} u_{23} u_{31},
u_{13} u_{21} u_{32},
u_{11} u_{23} u_{32},
u_{12} u_{21} u_{33},
u_{13} u_{22} u_{31}
\}$$
of polynomial ring
$K[u_{11},\ldots, u_{33}]$
and it is
a principal ideal
generated by $z_1z_2z_3-z_4z_5z_6$.
Given an integer $m \geq 1$,
$m$ multiple of ${\mathcal B}_3$ is defined by
$m {\mathcal B}_3 = \{ m \alpha  \ | \ \alpha \in {\mathcal B}_3 \}$.
Since it is well-known (due to Birkhoff) that
$$
m {\mathcal B}_3 \cap \ZZ^{3 \times 3}=
\{ \sigma_{i_1} + \cdots + \sigma_{i_m}
\ | \ 
1 \leq i_1, \ldots, i_m \leq 6
\},
$$
the toric ideal of $m {\mathcal B}_3$ is the toric ideal of
the nested configuration $A(B_1)$ where $A=\{t_1^m\}$.
In \cite{Haase}, they say that L. Piechnik and C. Haase proved 
that
the toric ideal of the multiple $2 n {\mathcal B}_3$ possesses a squarefree quadratic
initial ideal for $n > 1$.
This fact is directly obtained by Theorem \ref{maincase} since
the toric ideal of the multiple $2  {\mathcal B}_3$ possesses a squarefree quadratic
initial ideal.
Similarly,
since
the toric ideal of the multiple $3  {\mathcal B}_3$ possesses a squarefree quadratic
initial ideal, Theorem \ref{maincase} guarantees that 
the toric ideal of the multiple $3 n {\mathcal B}_3$ possesses a squarefree quadratic
initial ideal for $n > 1$.
However, since there are infinitely many prime numbers,
it is difficult to show the existence of a squarefree quadratic
initial ideal of the toric ideal of 
$m {\mathcal B}_3$ for all $m >1$ in this way.
In Theorem \ref{Birkhoff}, using another monomial order,
we will prove that  
the toric ideal of the multiple
$m {\mathcal B}_3$
possesses a quadratic Gr\"obner basis
for all $m > 1$.

In Section $4$, we give
a summary of our algebraic theory of nested configurations.

\section{Normality of toric rings of nested configurations}

The purpose of this section is to study
normality of $K[A(B_1,\ldots,B_d)]$.

\begin{Lemma}[\cite{Hoch}]
\label{Hoch}
The toric ring $K[A]$ is normal if and only if 
$$
\left\{
\left.
\frac{M_1}{M_2} 
\ \right| \ 
M_1, M_2 \in K[A] \mbox{ are monomials and }
\left(\frac{M_1}{M_2}\right)^m \in K[A]
\mbox{ for some } 0 < m \in \ZZ
\right\}
$$
is a subset of $K[A]$.
\end{Lemma}

\begin{Theorem}
\label{normality}
If $K[A]$, $K[B_1], \ldots,K[B_d]$ are normal,
then
$K[A(B_1,\ldots,B_d)]$ is normal.
\end{Theorem}

\begin{proof}
Suppose that $K[A]$, $K[B_1], \ldots,K[B_d]$ are normal
and that $K[A(B_1,\ldots,B_d)]$ is not normal.
Thanks to Lemma \ref{Hoch},
there exist monomials $M_1, M_2, M_3$
belonging to $K[A(B_1,\ldots,B_d)]$ such that
$M_1/M_2 \notin K[A(B_1,\ldots,B_d)]$ and that
$(M_1/M_2)^n = M_3$ for some integer $n > 1$.

Let $\psi :  K[A(B_1,\ldots,B_d)] \rightarrow K[A]$ be the surjective homomorphism
defined by $\psi ({m}_{j_1}^{(i_1)} \cdots {m}_{j_r}^{(i_r)} ) =  t_{i_1} \cdots t_{i_r} \in A$.
Then $\psi ( M_1), \psi(M_2) \in K[A]$
and
$$(\psi (M_1)/ \psi(M_2))^n = \psi(M_3) \in K[A].$$
Since $K[A]$ is normal, we have 
$\psi (M_1)  / \psi (M_2)   \in K[A]$.
Thus 
$\psi (M_1)/\psi (M_2) = \tb^{{\bf a}_{i_1}} \cdots \tb^{{\bf a}_{i_p}}$
for some $1 \leq i_1 , \ldots, i_p \leq n$.

Let
$\rho_k : K[A(B_1,\ldots,B_d)] \rightarrow K[B_k]$
be the surjective homomorphism defined by
$$
\rho_k(u_j^{(i)})
=
\left\{
\begin{array}{cl}
u_j^{(i)} & \mbox{if } i=k\\
 & \\
1 & \mbox{otherwise.}
\end{array}
\right.
$$
Then $\rho_k ( M_1), \rho_k (M_2) \in K[B_k]$
and
$(\rho_k (M_1)/ \rho_k (M_2))^n = \rho_k (M_3) \in K[B_k]$.
Since $K[B_k]$ is normal,
$\rho_k (M_1)/\rho_k (M_2) \in K[B_k]$.
Thus
$\rho_k (M_1)/\rho_k (M_2) = m_{j_1}^{(k)} \cdots m_{j_{q_k}}^{(k)} $
for some $1 \leq j_1 , \ldots, j_{q_k} \leq \lambda_k$.
Since $B_k$ is a configuration, it follows that
$\rho_k (M_1)= m_{u_1}^{(k)} \cdots m_{u_{q_k + r_k}}^{(k)}$ and
$\rho_k (M_2)=m_{v_1}^{(k)} \cdots m_{v_{r_k}}^{(k)}$.
Then
$\psi (M_1) = t_1^{q_1 + r_1} \cdots t_d^{q_d + r_d}$
and
$\psi (M_2) = t_1^{r_1} \cdots t_d^{r_d}$.
Thus we have
$$
\frac{\psi (M_1)}{\psi (M_2)} =\tb^{{\bf a}_{i_1}} \cdots \tb^{{\bf a}_{i_p}}
= t_1^{q_1} \cdots t_d^{q_d}.
$$
Hence 
$M_1/M_2 \in K[A(B_1,\ldots,B_d)]$ and this is a contradiction.
\end{proof}

The converse of Theorem \ref{normality} is false in general.

\begin{Example}
\label{ex13}
{\rm
Let $A = \{t_1^2\}$ and $B_1 = \{v, uv, u^3 v, u^4 v\}$.
Then $K[B_1]$ is not normal.
However, $I_{A(B_1)}$ has a squarefree quadratic initial ideal
and hence $K[A(B_1)] =K[\{u^i v^2 \ | \ i = 0,1,\ldots,8  \} ]$ is normal.
}
\end{Example}

Let $A =\{ \tb^{{\bf a}_1},\ldots,  \tb^{{\bf a}_n} \}$ be a configuration.
Then $K[A]$ is
called {\em very ample} if 
$$
(\ZZ \{ {\bf a}_1 ,\ldots, {\bf a}_n \}
\cap 
{\mathbb Q}_{\geq 0} \{ {\bf a}_1 ,\ldots, {\bf a}_n \})
\setminus
\ZZ_{\geq 0} \{ {\bf a}_1 ,\ldots, {\bf a}_n \}
$$
is a finite set.
In particular, $K[A]$ is very ample if $K[A]$ is normal.
Theorem \ref{normality} did not hold when we replaced ``normal" with ``very ample." 

\begin{Example}
{\rm
Let $A = \{t_1,t_2\}$, $B_1 = \{v, uv, u^3 v, u^4 v\}$ and $B_2=\{w\}$.
Then $K[A]$ and $K[B_2]$ are polynomial rings.
On the other hand, $K[B_1]$ is very ample, but not normal.
However, $K[A(B_1,B_2)] =K[v, uv, u^3 v, u^4 v,w]$ is not very ample.
In fact, the monomial $u^2 v w^\alpha$ does not belong
to $K[A(B_1,B_2)]$ for all $\alpha \in \ZZ_{\geq 0}$.
}
\end{Example}

Let $P_A$ denote the convex hull of $\{ {\bf a} \in {\mathbb Z}_{\geq 0}^d \ | \ {\bf t}^{\bf a} \in A \}$.
For a subset $B \subset A$,
$K[B]$ is called the
{\it combinatorial pure subring} (\cite{cpure, geom}) of $K[A]$
if there exists a face $F$ of $P_A$ such that 
$ 
\{ {\bf b} \in {\mathbb Z}_{\geq 0}^d \ | \ {\bf t}^{\bf b} \in B \}  =  \{ {\bf a} \in {\mathbb Z}_{\geq 0}^d \ | \ {\bf t}^{\bf a} \in A \}
\cap F.$
For example, if $B = A \cap K[ t_{i_1},\ldots,t_{i_s} ]$
for some $1 \leq i_1 < \cdots < i_s \leq d$, then
$K[B]$ is a combinatorial pure subring of $K[A]$.
(This is the original definition of a combinatorial pure subring
in \cite{cpure}.)

\begin{Lemma}
The toric ring $K[A(B_1,\ldots,B_d)]$ has a combinatorial pure subring
which is isomorphic to $K[A]$.
\end{Lemma}

\begin{proof}
For each $i = 1,2,\ldots,d$, let $\sigma_i$ be
an arbitrary monomial of $B_i$ which corresponds
to a vertex of $P_{B_i}$.
It follows that $K[A(\{\sigma_1\},\ldots,\{\sigma_d\})]$ is a
combintorial pure subring of $K[A(B_1,\ldots,B_d)]$.
Then $K[A(\{\sigma_1\},\ldots,\{\sigma_d\})] \simeq K[A]$.
\end{proof}

It is known \cite[Lemma 1]{veryample} that
every combinatorial pure subring of a normal (resp. very ample) semigroup ring 
is normal (resp. very ample).
Thus we have the following.

\begin{Theorem}
If
$K[A(B_1,\ldots,B_d)]$ is normal (resp. very ample),
then
$K[A]$ is normal (resp. very ample).
\end{Theorem}

\begin{Lemma}
\label{cpureB}
Let $m = \max (i \ | \ t_1^i t_2^{a_2} \cdots t_d^{a_d} \in A)  \geq 1$.
Then $K[A(B_1,\ldots,B_d)]$ has a combinatorial pure subring
which is isomorphic to $K[A'(B_1)]$ where $A'=\{ t_1^m \}$.
In particular, if $m =1$, then we have $K[A'(B_1)] \simeq K[B_1]$.
\end{Lemma}

\begin{proof}
Let $t_1^m t_2^{a_2} \cdots t_d^{a_d}$ be the largest monomial of $A$
with respect to a lexicographic order $t_1 > \cdots > t_d$.
Let $A = \{ \tb^{{\bf a}_1} = t_1^m t_2^{a_2} \cdots t_d^{a_d} ,  \tb^{{\bf a}_2},
\ldots, \tb^{{\bf a}_n}  \} $.
Thanks to \cite[Proposition 1.11]{Stu},
there exists a nonnegative integer vector ${\bf v}$ such that
${\bf v} \cdot {\bf a}_1 > {\bf v} \cdot {\bf a}_i$ for all $ 2 \leq i \leq n$.
Then $(m,a_2,\ldots,a_d)$ is a ${\bf v}$-vertex of $P_A$.
Hence $K[A(B_1,\ldots,B_d)]$ has a
combinatorial pure subring
$K[A''(B_1,\ldots,B_d)]$ with $A'' = \{t_1^m t_2^{a_2} \cdots t_d^{a_d}\}$.
For each $i = 2,\ldots,d$, let $\sigma_i$ be an arbitrary monomial of $B_i$ which corresponds
to a vertex of $P_{B_i}$.
It follows that $K[A''(B_1,\{\sigma_2\},\ldots,\{\sigma_d\})]$ is a
combinatorial pure subring of $K[A''(B_1,\ldots,B_d)]$.
Then $K[A''(B_1,\{\sigma_2\},\ldots,\{\sigma_d\})] \simeq K[A'(B_1)]$ where $A'=\{ t_1^m \}$.
\end{proof}

Thanks to Lemma \ref{cpureB}, we have the following.

\begin{Theorem}
If $A$ has no monomial divided by $t_i^2$ and if
$K[A(B_1,\ldots,B_d)]$ is normal (resp. very ample),
then
$K[B_i]$ is normal (resp. very ample).
\end{Theorem}

\begin{Corollary}
\label{introduction}
Suppose that a configuration $A$ consists of squarefree monomials.
Then
$K[A]$, $K[B_1], \ldots,K[B_d]$ are normal
if and only if
$K[A(B_1,\ldots,B_d)]$ is normal.
\end{Corollary}

\section{Gr\"obner bases of toric ideals of nested configurations}

In this section, using the technique (sorting operator)
in the proof of \cite[Theorem 14.2]{Stu},
we study Gr\"obner bases of
the toric ideal of a nested configuration.
The present section has three subsections:
\begin{itemize}
\item
Gr\"obner bases for polynomial ring case, i.e., each $K[B_i]$
is a polynomial ring;
\item
Gr\"obner bases for general case;
\item
Generators.
\end{itemize}
First, we introduce the sorting operator used in \cite{Stu}:

\begin{Example}[\cite{Stu}, Theorem 14.2]
\label{sturmfelssort}
{\em
Fix positive integers $r$ and $s_1, \ldots, s_d$.
Let
$$A = \{t_1^{i_1} \cdots t_d^{i_d}
 \ | \ 
i_1 + \cdots + i_d = r, \ 0 \leq i_1 \leq s_1, \ \ldots, \ 0 \leq i_d \leq s_d
\}.$$
We define a natural bijection between the element of $A$
and weakly increasing strings of length $r$ over the alphabet
$\{1,2,\ldots,d\}$ having at most $s_j$ occurrence of the letter $j$
which maps the monomial $t_1^{i_1} \cdots t_d^{i_d} \in A$ to
the weakly increasing string
$$
u_1 u_2 \cdots u_r = 
\underbrace{1 1 \cdots 1}_{i_1 \mbox{ {\tiny times}}}
\underbrace{2 2 \cdots 2}_{i_2 \mbox{ {\tiny times}}}
\underbrace{3 3 \cdots 3}_{i_3 \mbox{ {\tiny times}}} \cdots 
\underbrace{d d \cdots d}_{i_d \mbox{ {\tiny times}}}
.$$
We write $x_{u_1 u_2 \cdots u_r}$ for the corresponding variable in $K[\xb]$.
Let ${\rm sort} (\cdot)$ denote the operator which takes any string over
the alphabet
$\{1,2,\ldots,d\}$ 
and sorts it into weakly increasing order.
It is known \cite[Theorem 14.2]{Stu} that
there exists a monomial order $<$ on $K[\xb]$ such that
$$
\{
x_{u_1 u_2 \cdots u_r} x_{v_1 v_2 \cdots v_r}
-
x_{w_1 w_3 \cdots w_{2r-1}} x_{w_2 w_4 \cdots w_{2r}}
\ | \ 
w_1 w_2 w_3 \cdots w_{2r} = {\rm sort} (u_1v_1u_2v_2 \cdots u_rv_r) 
\}
$$
is a quadratic Gr\"obner basis of $I_A$ with respect to $<$
and ${\rm in}_< (I_A)$ is squarefree.
For example, $x_{1 2} x_{3 3} - x_{1 3} x_{2 3}$ belongs to 
the Gr\"obner basis since we have
$1 2 3 3 = {\rm sort} (1 3 2 3)$.
}
\end{Example}

Let, as before, $A=\{\tb^{\ab_1},\ldots,\tb^{\ab_n}\}$ and
$B_i=\{m_1^{(i)},\ldots,m_{\lambda_i}^{(i)}\}$
for $1 \leq i  \leq d$.
Let $K[{\bf x}]$ be a polynomial ring with the set of variables
$$
\left\{x_{(i_1,j_1) \cdots (i_r,j_r) }^{(k)} 
\ \left| \ 
\begin{array}{c}
1 \leq i_1 \leq \cdots \leq i_r \leq d, \ 
1 \leq k \leq n\\
t_{i_1} \cdots t_{i_r} = \tb^{\ab_k} \in A\\  
{m}_{j_1}^{(i_1)} \cdots {m}_{j_r}^{(i_r)}  \in A(B_1,\ldots , B_d)
\end{array}
\right.
\right\}
$$
and let
$K[{\bf y}] = K[y_1,\ldots,y_n]$ and 
$K\left[{\bf z}^{(i)}\right] = K\left[z_1^{(i)},\ldots , z_{\lambda_i}^{(i)}\right] $
$(i= 1,2,\ldots,d)$
be polynomial rings.
The toric ideal $I_{A}$ is the kernel of the
homomorphism
$\pi_0 : K[{\bf y}] \longrightarrow K[{\bf t}]$
defined by setting $\pi_0  (y_k)=\tb^{\ab_k}.$
The toric ideal $I_{B_i}$ is the kernel of the
homomorphism
$\pi_{i} : K[{\bf z}^{(i)}] \longrightarrow K[{\bf u}^{(i)}]$
defined by setting $\pi_{i} (z_j^{(i)}) =m_j^{(i)}.$
The toric ideal $I_{A(B_1,\ldots,B_d) }$ is the kernel of the
homomorphism
$\pi : K[{\bf x}] \longrightarrow K[{\bf u}^{(1)},\ldots,{\bf u}^{(d)} ]$
defined by setting
$\pi\left(x_{(i_1,j_1) \cdots (i_r,j_r) }^{(k)}\right)={m}_{j_1}^{(i_1)} \cdots {m}_{j_r}^{(i_r)}.$

\begin{Lemma}
\label{blemma}
Let 
$
p_1 = x_{(i_1,j_1) \cdots (i_r,j_r) }^{(k)} x_{(i_{r+1},j_{r+1}) \cdots (i_{2r},j_{2r}) }^{(k)}
$
be a quadratic monomial in $K[{\bf x}]$
and
let ${\rm sort} (\cdot)$ be the sorting operator
over the alphabet 
$$
\{
(1,1), (1,2),\ldots,(1,\lambda_1),(2,1), \ldots,(d,\lambda_d)
\}
$$
with respect to the ordering
$$
(1,1) \succ  (1,2) \succ \cdots \succ (1,\lambda_1) \succ (2,1) \succ \cdots \succ (d,\lambda_d).
$$
Then,
$
p_2 = x_{(i_1',j_1')(i_3',j_3') \cdots (i_{2r-1}',j_{2r-1}') }^{(k)} x_{(i_2',j_2')(i_4',j_4') \cdots (i_{2r}',j_{2r}') }^{(k)}
$
where
$$(i_1',j_1') \cdots (i_{2r}',j_{2r}') = {\rm sort} ( (i_1,j_1) \cdots (i_{2r},j_{2r}))$$
is a monomial belonging to $K[{\bf x}]$ and, in particular, we have
$
p_1-p_2 \in I_{A(B_1,\ldots,B_d)}
$.
\end{Lemma}

\begin{proof}
Suppose that $x_{(i_1',j_1')(i_3',j_3') \cdots (i_{2r-1}',j_{2r-1}') }^{(k)}$
is not a variable in $K[{\bf x}]$.
Then we have $t_{i_1'} t_{i_3'} \cdots t_{i_{2r-1}'} \neq \tb^{\ab_k}$ and 
hence there exist integers $1 \leq i \leq d$ and $\alpha$
such that $t_i^\alpha$ divides $\tb^{\ab_k}$ and does not divide
$ t_{i_1'} t_{i_3'} \cdots t_{i_{2r-1}'}$.
Since $i_1' \leq \cdots \leq i_{2r}'$,
it then follows that $t_i^{2 \alpha}$ does not divide
$t_{i_1'} t_{i_2'} \cdots t_{i_{2r}'}$.
Thanks to
$(i_1',j_1') \cdots (i_{2r}',j_{2r}') = {\rm sort} ( (i_1,j_1) \cdots (i_{2r},j_{2r}))$,
we have $t_{i_1} t_{i_2} \cdots t_{i_{2r}} = t_{i_1'} t_{i_2'} \cdots t_{i_{2r}'}$.
Hence $t_i^{2 \alpha}$ does not divide $t_{i_1} t_{i_2} \cdots t_{i_{2r}}$.
It follows that $t_i^{\alpha}$ does not divide either 
$t_{i_1} t_{i_2} \cdots t_{i_{r}}$ or $t_{i_{r+1}} t_{i_{r+2}} \cdots t_{i_{2r}}$.
Thus either $t_{i_1} t_{i_2} \cdots t_{i_{r}}$ or $t_{i_{r+1}} t_{i_{r+2}} \cdots t_{i_{2r}}$
is not equal to $\tb^{\ab_k}$.
This contradicts that $p_1$ is a monomial of $K[\xb]$.

On the other hand, by virtue of 
$(i_1',j_1') \cdots (i_{2r}',j_{2r}') = {\rm sort} ( (i_1,j_1) \cdots (i_{2r},j_{2r}))$,
we have $\pi(p_1) =\pi(p_2)$ and hence
$p_1-p_2 \in I_{A(B_1,\ldots,B_d)}
$ as desired.
\end{proof}

\begin{Lemma}
\label{alemma}
Let $y_{k_1} \cdots y_{k_p}
-
y_{k_1'} \cdots y_{k_p'}$ be a binomial in $I_A$ and
let $$
\prod_{\ell=1}^p
x_{
(i_{(\ell-1)r+1},j_{(\ell-1)r+1}) \cdots (i_{\ell r},j_{\ell r})
}^{(k_\ell)}
$$
be a monomial in $K[{\bf x}]$.
Then, there exists a binomial
$$
\prod_{\ell=1}^p
x_{
(i_{(\ell-1)r+1},j_{(\ell-1)r+1}) \cdots (i_{\ell r},j_{\ell r})
}^{(k_\ell)}
-
\prod_{\ell=1}^p
x_{
(i_{(\ell-1)r+1}',j_{(\ell-1)r+1}') \cdots (i_{\ell r}',j_{\ell r}')
}^{(k_\ell')}
\in
I_{A(B_1,\ldots,B_d)},$$
where
$
{\rm sort}
(
(i_1,j_1) \cdots  (i_{p r},j_{p r})
)
=
{\rm sort}
(
(i_1',j_1') \cdots  (i_{p r}',j_{p r}')
)
$.
\end{Lemma}

\begin{proof}
Let $\pi_0(y_{k_\ell}') = 
t_{i_{(\ell-1)r+1}'} \cdots t_{i_{\ell r}'} 
$
for each $1 \leq \ell \leq p$.
Since $y_{k_1} \cdots y_{k_p}
-
y_{k_1'} \cdots y_{k_p'}$ belongs to $I_A$,
we have $\prod_{\ell=1}^{pr} t_{i_\ell} =  \prod_{\ell=1}^{pr} t_{i_\ell'}$.
Hence there exist $j_1',\ldots,j_{pr}'$ such that
$$
{\rm sort}
(
(i_1,j_1) \cdots  (i_{p r},j_{p r})
)
=
{\rm sort}
(
(i_1',j_1') \cdots  (i_{p r}',j_{p r}')
)
.$$
It then follows that
$$
\prod_{\ell=1}^p
x_{
(i_{(\ell-1)r+1},j_{(\ell-1)r+1}) \cdots (i_{\ell r},j_{\ell r})
}^{(k_\ell)}
-
\prod_{\ell=1}^p
x_{
(i_{(\ell-1)r+1}',j_{(\ell-1)r+1}') \cdots (i_{\ell r}',j_{\ell r}')
}^{(k_\ell')}
\in
I_{A(B_1,\ldots,B_d)}$$
as desired.
\end{proof}

Fix a monomial order $<_i$ on $K[{\bf z}^{(i)}]$ for each $1 \leq i \leq d$.
Let ${\mathcal G}_i$ be a Gr\"obner basis of $I_{B_i}$ with respect to $<_i$.
For each $M \in A(B_1,\ldots,B_d)$,
the expression $M = {m}_{j_1}^{(i_1)} \cdots {m}_{j_r}^{(i_r)}$ is called
{\it standard}
if
$$\prod_{i_\ell = j , \ \  1 \leq \ell \leq r} z_{j_\ell}^{(i_\ell)} $$
is a standard monomial with respect to ${\mathcal G}_j$
for all $1 \leq j \leq d$.
In order to study the relation among $I_{A}$, $I_{B_i}$ and $I_{A(B_1,\ldots,B_d) }$,
we define homomorphisms
\begin{eqnarray*}
\varphi_0 : K[{\bf x}] \longrightarrow K[{\bf y}] \ \ &,& 
\varphi_0 \left(x_{(i_1,j_1) \cdots (i_r,j_r)}^{(k)}\right) = y_k, \\
\varphi_j : K[{\bf x}] \longrightarrow K[{\bf z}^{(j)}] &,&  
\varphi_j \left(x_{(i_1,j_1) \cdots (i_r,j_r)}^{(k)}\right) = 
\prod_{i_\ell = j , \ \  1 \leq \ell \leq r} z_{j_\ell}^{(i_\ell)},
\end{eqnarray*}
where ${m}_{j_1}^{(i_1)} \cdots {m}_{j_r}^{(i_r)}$ is the standard expression
defined above.

\begin{Lemma}[\cite{AHOT2}]
\label{keylemma}
Let $f$ be a binomial in $K[{\bf x}]$.
Then 
$f \in I_{A(B_1,\ldots,B_d) }$
if and only if
$\varphi_i(f) \in I_{B_i}$ for all $1 \leq i \leq d$.
Moreover, if $f$ belongs to $I_{A(B_1,\ldots,B_d) }$, then we have
$\varphi_0 (f) \in I_A$.
\end{Lemma}

\subsection{Polynomial ring case}

First, we study the case when all of $K[B_i]$ are polynomial rings.

\begin{Theorem}
\label{pcase}
Let ${\mathcal G}_0$ be a Gr\"obner basis of $I_A$ with respect to a monomial order
$<_0$.
If each $B_i$ is a set of variables,
then the toric ideal $I_{A(B_1,\ldots,B_d) }$
possesses a Gr\"obner basis consisting of the following binomials:
\begin{enumerate}
\item[(1)]
$\displaystyle
\underline
{
\prod_{\ell=1}^p
x_{
(i_{(\ell-1)r+1},j_{(\ell-1)r+1}) \cdots (i_{\ell r},j_{\ell r})
}^{(k_\ell)}
}
-
\prod_{\ell=1}^p
x_{
(i_{(\ell-1)r+1}',j_{(\ell-1)r+1}') \cdots (i_{\ell r}',j_{\ell r}')
}^{(k_\ell')}
$

\medskip

\noindent
where
$
\underline{
y_{k_1}
\cdots
y_{k_p}
}
-
y_{k_1'}
\cdots
y_{k_p'}
\in {\mathcal G}_0
$
and
$$
{\rm sort}
(
(i_1,j_1) \cdots  (i_{p r},j_{p r})
)
=
{\rm sort}
(
(i_1',j_1') \cdots  (i_{p r}',j_{p r}')
).
$$

\item[(2)]
$
\underline{
x_{(i_1,j_1) \cdots (i_r,j_r) }^{(k)} x_{(i_{r+1},j_{r+1}) \cdots (i_{2r},j_{2r}) }^{(k)}
}
-
x_{(i_1',j_1')(i_3',j_3') \cdots (i_{2r-1}',j_{2r-1}') }^{(k)} x_{(i_2',j_2')(i_4',j_4') \cdots (i_{2r}',j_{2r}') }^{(k)}
$

\medskip

\noindent
where 
${\rm sort} ( (i_1,j_1) \cdots (i_{2r},j_{2r})) =   (i_1',j_1') \cdots (i_{2r}',j_{2r}')$
with respect to the ordering
$
(1,1) \succ  (1,2) \succ \cdots \succ (1,\lambda_1) \succ (2,1) \succ \cdots \succ (d,\lambda_d)
$.
\item[(3)]
$
\underline{
x_{(i_1,j_1)\cdots(i_\ell,j_\ell) \cdots (i_r,j_r) }^{(k)}
x_{(i_1',j_1')\cdots(i_{\ell'}',j_{\ell'}') \cdots (i_r',j_r') }^{(k')}
}
-
x_{(i_1,j_1)\cdots(i_{\ell'}',j_{\ell'}')\cdots (i_r,j_r) }^{(k)}
x_{(i_1',j_1')\cdots(i_\ell,j_\ell) \cdots (i_r',j_r') }^{(k')} 
$

\medskip

\noindent
where 
$k< k'$, $i_\ell = i_{\ell'}'$ and $j_\ell > j_{\ell'}'$.

\end{enumerate}

\bigskip

\noindent
The initial monomial
of each binomial is the first (underlined) monomial and,
in particular,
the initial monomial of each binomial in (2) and (3)
is squarefree.
Moreover, the initial monomial of each binomial in (1)
is squarefree (resp. quadratic) if
the corresponding monomial $y_{k_1}
\cdots
y_{k_p}$
is squarefree (resp. quadratic).
\end{Theorem}

\begin{proof}
Let ${\mathcal G}$ denote the set of binomials above.
Thanks to Lemmas \ref{blemma} and \ref{alemma},
it is easy to see that
${\mathcal G}$ is a (finite) subset of $I_{A(B_1,\ldots,B_d)}$.

\bigskip

\noindent
{\bf Claim 1.}
There exists a monomial order such that the initial monomial of 
each binomial in ${\mathcal G}$ is the underlined monomial.

By virtue of \cite[Theorem 3.12]{Stu},
it is enough to show that the reduction modulo ${\mathcal G}$ is Noetherian.
Suppose that there exists a sequence of reductions modulo ${\mathcal G}$ which
does not terminate.
Let $v$ be a monomial in $K[\xb]$ and assume
$v \stackrel{g}{\longrightarrow} v'$ with $g \in {\mathcal G}$. 
Then we have
$$
\left\{
\begin{array}{cl}
\varphi_0(v) >_0 \varphi_0(v') & \ \ \mbox{if } g \mbox{ in (1)},\\
\varphi_0(v) = \varphi_0(v') & \ \ \mbox{otherwise.} 
\end{array}
\right.
$$
Hence the number of binomials in (1) appearing in the sequence is finite.
Thus we may assume that the binomials in (1) do not appear in
the sequence.
Let $v$ be a monomial in $K[\xb]$ and assume
$v \stackrel{g}{\longrightarrow} v'$ where $g \in {\mathcal G}$ belongs to
either (2) or (3). 
Since $g$ belongs to either (2) or (3), $v$ and $v'$ is of the form
$
v=
\prod_{\ell=1}^p
x_{
(i_{(\ell-1)r+1},j_{(\ell-1)r+1}) \cdots (i_{\ell r},j_{\ell r})
}^{(k_\ell)},
v'=\prod_{\ell=1}^p
x_{
(i_{(\ell-1)r+1}',j_{(\ell-1)r+1}') \cdots (i_{\ell r}',j_{\ell r}')
}^{(k_\ell)}
$.
Let
$$
{\rm Inversion}(v)=
\left\{
(\xi,\xi')   \ \left| \ 
\begin{array}{c}
\ell(r-1) +1 \leq  \xi \leq \ell r\\
\ell'(r-1) +1 \leq  \xi' \leq \ell' r\\
i_\xi = i_{\xi'}, \ j_\xi > j_{\xi'}\\
k_\ell < k_{\ell'}
\end{array}
\right.
\right\},
$$
$$
{\rm Inversion}(v')=
\left\{
(\xi,\xi')   \ \left| \ 
\begin{array}{c}
\ell(r-1) +1 \leq  \xi \leq \ell r\\
\ell'(r-1) +1 \leq  \xi' \leq \ell' r\\
i_\xi' = i_{\xi'}', \ j_\xi' > j_{\xi'}'\\
k_\ell < k_{\ell'}
\end{array}
\right.
\right\}.
$$
Then the cardinality of these sets satisfies
$
\sharp \left| {\rm Inversion}(v) \right|
\geq
\sharp \left| {\rm Inversion}(v') \right|
$
where equality holds if and only if $g$ belongs to (2).
Hence the number of binomials in (3) appearing in the sequence is finite.
Thus we may assume that the binomials in (3) do not appear in
the sequence.
However, any sequence of reductions modulo the set of binomials
in (2) corresponds to the sort of the indices and hence it terminates.
This is a contradiction.

\bigskip

\noindent
{\bf Claim 2.}
The set ${\mathcal G}$ is a Gr\"obner basis of $I_{A(B_1,\ldots,B_d)}$.

Suppose that ${\mathcal G}$ is not a Gr\"obner basis of $I_{A(B_1,\ldots,B_d)}$.
Thanks to Lemmas \ref{blemma} and \ref{alemma},
there exists a binomial $f=p_1 -p_2 \in I_{A(B_1,\ldots,B_d)}$ such that
neither $p_1$ nor $p_2$ is divisible by the initial monomial of 
any binomial in ${\mathcal G}$.
By virtue of Lemma \ref{keylemma}, we have
$\varphi_0 (f) =  \varphi_0(p_1) - \varphi_0(p_2) \in I_A$.
If $\varphi_0(p_1) - \varphi_0(p_2) \neq 0$, then there exists a binomial
$g \in {\mathcal G}_0$ such that the initial monomial of $g$
divides either $\varphi_0(p_1)$ or $\varphi_0(p_2)$.
This contradicts that neither $p_1$ nor $p_2$ is divisible by the initial monomial of 
any binomial in (1).
Hence we have $\varphi_0(p_1) = \varphi_0(p_2)$.
Thus $f$ is of the form
$$
f=
\prod_{\ell=1}^p
x_{
(i_{(\ell-1)r+1},j_{(\ell-1)r+1}) \cdots (i_{\ell r},j_{\ell r})
}^{(k_\ell)}
-\prod_{\ell=1}^p
x_{
(i_{(\ell-1)r+1}',j_{(\ell-1)r+1}') \cdots (i_{\ell r}',j_{\ell r}')
}^{(k_\ell)}
.$$
Since
neither $p_1$ nor $p_2$ is divisible by the initial monomial of 
any binomial in either (2) or (3),
it follows that
$p_1 = p_2$ and hence $f=0$.
\end{proof}

\subsection{General case}

We now study the general case.

\begin{Theorem}
\label{maincase}
Let ${\mathcal G}_0$ be a Gr\"obner basis of $I_A$
and let ${\mathcal G}_i$ be a Gr\"obner basis of $I_{B_i}$ with respect to $<_i$.
Then the toric ideal $I_{A(B_1,\ldots,B_d) }$
possesses a Gr\"obner basis consisting of the binomials 
(1), (2) and (3)
appearing in
Theorem \ref{pcase} together with the following binomials:
\begin{enumerate}
\item[(4)]
$
\underline
{
\prod_{\ell=1}^p
x_{
M_\ell (i,j_{\ell,1}) \cdots (i,j_{\ell,q_\ell}) M_\ell'
}^{(k_\ell)}
}
-
\prod_{\ell=1}^p
x_{
M_\ell (i,j_{\ell,1}') \cdots (i,j_{\ell,q_\ell}') M_\ell' 
}^{(k_\ell)}
$
where the binomial\\
$
\displaystyle
0 \neq \underline{
\prod_{\ell=1}^p
z_{j_{\ell,1}}^{(i)}
\cdots
z_{j_{\ell,q_\ell}}^{(i)}
}
-
\prod_{\ell=1}^p
z_{j_{\ell,1}'}^{(i)}
\cdots
z_{j_{\ell,q_\ell}'}^{(i)}$
belongs to 
${\mathcal G}_i
$.
\end{enumerate}
The initial monomial
of each binomial is the first (underlined) monomial and,
in particular,
the initial monomial of each binomial above
is squarefree (resp. quadratic) if
the corresponding monomial 
$
\prod_{\ell=1}^p
z_{j_{\ell,1}}^{(i)}
\cdots
z_{j_{\ell,q_\ell}}^{(i)}
$
is squarefree (resp. quadratic).
\end{Theorem}

\begin{proof}
Let ${\mathcal G}$ denote the set of binomials above.
Thanks to Lemmas \ref{blemma}, \ref{alemma} and \ref{keylemma},
${\mathcal G}$ is a (finite) subset of $I_{A(B_1,\ldots,B_d)}$.

\bigskip

\noindent
{\bf Claim 1.}
There exists a monomial order such that the initial monomial of 
each binomial in ${\mathcal G}$ is the underlined monomial.

By virtue of \cite[Theorem 3.12]{Stu},
it is enough to show that the reduction modulo ${\mathcal G}$ is Noetherian.
Suppose that there exists a sequence of reductions modulo ${\mathcal G}$ which
does not terminate.
Let $v$ be a monomial in $K[\xb]$ and assume
$v \stackrel{g}{\longrightarrow} v'$ with $g \in {\mathcal G}$. 
Then we have
$$
\left\{
\begin{array}{cl}
\varphi_j(v) >_j \varphi_j(v') & \ \ \mbox{if } g \mbox{ is in (4) and arising from } {\mathcal G}_j ,\\
\varphi_j(v) = \varphi_j(v') & \ \ \mbox{otherwise.} 
\end{array}
\right.
$$
Hence the number of binomials in (4) appearing in the sequence is finite.
Thus we may assume that the binomials in (4) do not appear in
the sequence.
However, as we proved in the proof of Theorem \ref{pcase},
there exists no sequence of reductions modulo the set of binomials
in (1), (2) and (3) which does not terminate.
This is a contradiction.

\bigskip

\noindent
{\bf Claim 2.}
The set ${\mathcal G}$ is a Gr\"obner basis of $I_{A(B_1,\ldots,B_d)}$.

Suppose that ${\mathcal G}$ is not a Gr\"obner basis of $I_{A(B_1,\ldots,B_d)}$.
Thanks to Lemmas \ref{blemma}, \ref{alemma} and \ref{keylemma},
there exists a binomial $f=p_1 -p_2 \in I_{A(B_1,\ldots,B_d)}$ such that
neither $p_1$ nor $p_2$ is divisible by the initial monomial of 
any binomial in ${\mathcal G}$.
By virtue of Lemma \ref{keylemma}, we have
$\varphi_i(f)=\varphi_i(p_1) - \varphi_i(p_2) \in I_{B_i}$ for all $1 \leq i \leq d$.
If $\varphi_i(p_1) - \varphi_i(p_2) \neq 0$ for some $i$,
then there exists a binomial
$g' \in {\mathcal G}_i$ such that the initial monomial of $g'$
divides either $\varphi_i(p_1)$ or $\varphi_i(p_2)$.
This contradicts that neither $p_1$ nor $p_2$ is divisible by the initial monomial of 
any binomial in (4).
Hence we have $\varphi_i(p_1) = \varphi_i(p_2)$ for all $i$.
Moreover, thanks to the argument in the proof of Theorem \ref{pcase},
we have $\varphi_0(p_1) = \varphi_0(p_2)$.

Thus $f$ is of the form
$$
f=
\prod_{\ell=1}^p
x_{
(i_{(\ell-1)r+1},j_{(\ell-1)r+1}) \cdots (i_{\ell r},j_{\ell r})
}^{(k_\ell)}
-\prod_{\ell=1}^p
x_{
(i_{(\ell-1)r+1}',j_{(\ell-1)r+1}') \cdots (i_{\ell r}',j_{\ell r}')
}^{(k_\ell)}
,$$
where ${\rm sort} ( (i_1,j_1) \cdots (i_{pr},j_{pr})) =   {\rm sort} ((i_1',j_1') \cdots (i_{pr}',j_{pr}'))$.
Since
neither $p_1$ nor $p_2$ is divisible by the initial monomial of 
any binomial in either (2) or (3),
it follows that
$p_1 = p_2$ and hence $f=0$.
\end{proof}

If ${\mathcal G}_i$ possesses a binomial of degree $3$,
then we need the following binomials:

\begin{enumerate}
\item[(a)]

$
x_{ M_1 (i,j_1) M_1' }^{(k_1)}
x_{ M_2 (i,j_2) M_2' }^{(k_2)}
x_{ M_3 (i,j_3) M_3' }^{(k_3)}
-
x_{ M_1 (i,j_1') M_1' }^{(k_1)}
x_{ M_2 (i,j_2') M_2' }^{(k_2)}
x_{ M_3 (i,j_3') M_3' }^{(k_3)}
$\\
where 
$
z_{j_1}^{(i)} z_{j_2}^{(i)}z_{j_3}^{(i)} 
- 
z_{j_1'}^{(i)} z_{j_2'}^{(i)}z_{j_3'}^{(i)} 
\in {\mathcal G}_i$.
\item[(b)]
$
x_{ M_1 (i,j_1)(i,j_2) M_1' }^{(k_1)}
x_{ M_2 (i,j_3) M_2' }^{(k_2)}
-
x_{ M_1 (i,j_1')(i,j_2') M_1' }^{(k_1)}
x_{ M_2 (i,j_3') M_2' }^{(k_2)}
$\\
where 
$z_{j_1}^{(i)} z_{j_2}^{(i)}z_{j_3}^{(i)} - z_{j_1'}^{(i)} z_{j_2'}^{(i)}z_{j_3'}^{(i)} 
 \in {\mathcal G}_i$.
\end{enumerate}
We do not need (b) if $A$ has no monomial divided by $t_i^2$.
In general, we have
$$
\deg\left(
\prod_{\ell=1}^p
z_{j_{\ell,1}}^{(i)}
\cdots
z_{j_{\ell,q_\ell}}^{(i)}
\right)
=
\sum_{\ell=1}^p
q_\ell \geq p =
\deg
\left(
\prod_{\ell=1}^p
x_{
M_\ell (i,j_{\ell,1}) \cdots (i,j_{\ell,q_\ell}) M_\ell'
}^{(k_\ell)}
\right).
$$

\medskip

\noindent
The binomials of type (a) are not always needed
for a minimal Gr\"obner basis even if
${\mathcal G}_i$ has a cubic binomial.
In such a case, $I_{A(B_1,\ldots,B_d)}$ may have a quadratic Gr\"obner basis.
In Section 3, we will show an example.

\subsection{Generators}

Thanks to a part of the argument in the proof of Theorem \ref{maincase},
we have the following.

\begin{Proposition}
Let ${\mathcal H}_0$ be a set of binomial generators of $I_A$
and
let ${\mathcal H}_i$ be a set of binomial generators of $I_{B_i}$.
Then, the toric ideal $I_{A(B_1,\ldots,B_d) }$
is generated by the following binomials:
\begin{enumerate}
\item[(1)]
$\displaystyle
\prod_{\ell=1}^p
x_{
(i_{(\ell-1)r+1},j_{(\ell-1)r+1}) \cdots (i_{\ell r},j_{\ell r})
}^{(k_\ell)}
-
\prod_{\ell=1}^p
x_{
(i_{(\ell-1)r+1}',j_{(\ell-1)r+1}') \cdots (i_{\ell r}',j_{\ell r}')
}^{(k_\ell')}
$

\medskip

\noindent
where
$
y_{k_1}
\cdots
y_{k_p}
-
y_{k_1'}
\cdots
y_{k_p'}
\in {\mathcal H}_0
$
and
$$
{\rm sort}
(
(i_1,j_1) \cdots  (i_{p r},j_{p r})
)
=
{\rm sort}
(
(i_1',j_1') \cdots  (i_{p r}',j_{p r}')
).
$$
\item[(2)]
$
x_{(i_1,j_1) \cdots (i_r,j_r) }^{(k)} x_{(i_{r+1},j_{r+1}) \cdots (i_{2r},j_{2r}) }^{(k)}
-
x_{(i_1',j_1')(i_3',j_3') \cdots (i_{2r-1}',j_{2r-1}') }^{(k)} x_{(i_2',j_2')(i_4',j_4') \cdots (i_{2r}',j_{2r}') }^{(k)}
$

\medskip

\noindent
where 
${\rm sort} ( (i_1,j_1) \cdots (i_{2r},j_{2r})) =   (i_1',j_1') \cdots (i_{2r}',j_{2r}')$
with respect to the ordering
$
(1,1) \succ  (1,2) \succ \cdots \succ (1,\lambda_1) \succ (2,1) \succ \cdots \succ (d,\lambda_d)
$.
\item[(3)]
$
x_{(i_1,j_1)\cdots(i_\ell,j_\ell) \cdots (i_r,j_r) }^{(k)}
x_{(i_1',j_1')\cdots(i_{\ell'}',j_{\ell'}') \cdots (i_r',j_r') }^{(k')}
-
x_{(i_1,j_1)\cdots(i_{\ell'}',j_{\ell'}')\cdots (i_r,j_r) }^{(k)}
x_{(i_1',j_1')\cdots(i_\ell,j_\ell) \cdots (i_r',j_r') }^{(k')} 
$

\medskip

\noindent
where 
$k< k'$, $i_\ell = i_{\ell'}'$ and $j_\ell > j_{\ell'}'$.
\item[(4)]
$
\prod_{\ell=1}^p
x_{
M_\ell (i,j_{\ell,1}) \cdots (i,j_{\ell,q_\ell}) M_\ell'
}^{(k_\ell)}
-
\prod_{\ell=1}^p
x_{
M_\ell (i,j_{\ell,1}') \cdots (i,j_{\ell,q_\ell}') M_\ell' 
}^{(k_\ell)}
$
where the binomial\\
$
\displaystyle
0 \neq
\prod_{\ell=1}^p
z_{j_{\ell,1}}^{(i)}
\cdots
z_{j_{\ell,q_\ell}}^{(i)}
-
\prod_{\ell=1}^p
z_{j_{\ell,1}'}^{(i)}
\cdots
z_{j_{\ell,q_\ell}'}^{(i)}$
belongs to 
${\mathcal H}_i
$.
\end{enumerate}

\end{Proposition}

\section{Toric ideals of multiples of the Birkhoff polytope}

Let
${\bf c} = (c_1,c_2,c_3) \in \ZZ_{\geq 0}^3$ and
${\bf r} = (r_1,r_2,r_3) \in \ZZ_{\geq 0}^3$ be
vectors with $c_1+c_2+c_3=r_1+r_2+r_3$.
Then $3 \times 3$ {\em transportation polytope} $T_{{\bf r}{\bf c}}$
is the set of all non-negative $3 \times 3$ matrices $A=(a_{ij})$
satisfying
$$
\sum_{i=1}^3 a_{ik} = c_k \mbox{ and } 
\sum_{j=1}^3 a_{\ell j} = r_\ell
$$
for $1 \leq k, \ell \leq 3$.
It is known that
this is a bounded convex polytope of dimension 4
whose vertices are
lattice points in $\RR^{3 \times 3}$.
The toric ideal of  $T_{{\bf r}{\bf c}}$ is the toric ideal of
the configuration
$
\{ \tb^\alpha \ | \ \alpha \in T_{{\bf r}{\bf c}} \cap  \ZZ^{3 \times 3}\}
.$

\begin{Example}
{\rm
Let ${\bf c} = {\bf r} = (1,1,1)$.
Then the transportation polytope ${\mathcal B}_3 := T_{{\bf r}{\bf c}}$
is called the {\em Birkhoff polytope}.
The lattice points in ${\mathcal B}_3$ are 
$$
\sigma_1=
\left(
\begin{array}{ccc}
1 & 0 & 0\\
0 & 1 & 0\\
0 & 0 & 1
\end{array}
\right),
\sigma_2=
\left(
\begin{array}{ccc}
0 & 1 & 0\\
0 & 0 & 1\\
1 & 0 & 0
\end{array}
\right),
\sigma_3=
\left(
\begin{array}{ccc}
0 & 0 & 1\\
1 & 0 & 0\\
0 & 1 & 0
\end{array}
\right),
$$
$$
\sigma_4=
\left(
\begin{array}{ccc}
1 & 0 & 0\\
0 & 0 & 1\\
0 & 1 & 0
\end{array}
\right),
\sigma_5=
\left(
\begin{array}{ccc}
0 & 1 & 0\\
1 & 0 & 0\\
0 & 0 & 1
\end{array}
\right),
\sigma_6=
\left(
\begin{array}{ccc}
0 & 0 & 1\\
0 & 1 & 0\\
1 & 0 & 0
\end{array}
\right).
$$
The toric ideal of ${\mathcal B}_3$ is the toric ideal
of the configuration
$$
\{
u_{11} u_{22} u_{33},
u_{12} u_{23} u_{31},
u_{13} u_{21} u_{32},
u_{11} u_{23} u_{32},
u_{12} u_{21} u_{33},
u_{13} u_{22} u_{31}
\}$$
and it is
a principal ideal
generated by $z_1z_2z_3-z_4z_5z_6$.
}
\end{Example}

The following is proved by Haase--Paffenholz \cite{Haase}:
\begin{itemize}
\item
The toric ideal of
$3 \times 3$ transportation polytope is generated by quadratic binomials
except for
${\mathcal B}_3$.
\item
The toric ideal of
$3 \times 3$ transportation polytope possesses a quadratic squarefree
initial ideal if it is not a multiple of ${\mathcal B}_3$.
\end{itemize}
Thus, it is natural to ask whether
the toric ideal of a multiple of ${\mathcal B}_3$ possesses a quadratic
Gr\"obner basis except for ${\mathcal B}_3$.
The following fact is due to Birkhoff:
\begin{itemize}
\item
Every non-negative integer $p \times p$ matrix with equal
row and column sums can be written as a sum of permutation matrices.
\end{itemize}
Hence, in particular, we have
$$
n {\mathcal B}_3 \cap \ZZ^{3 \times 3}=
\{ \sigma_{i_1} + \cdots + \sigma_{i_n}
\ | \ 
1 \leq i_1, \ldots, i_n \leq 6
\}.
$$
Thus, in order to study the toric ideal of $n$ multiple of ${\mathcal B}_3$,
we consider the following:

\begin{Example}
\label{bexample}
{\em
Let
$A = \{t_1^n \}$ and suppose that $B_1$
satisfies $\sharp|B_1|=6$ and $I_{B_1}= \left<z_1z_2z_3-z_4z_5z_6 \right>$.
If $n=1$, then $A(B_1)=B_1$ and
$\{x_1 x_2 x_3 -x_4 x_5 x_6
\}$ is the reduced Gr\"obner basis of
$I_{A(B_1)}$ with respect to any monomial order.
If $n >1$, then, by virtue of Theorem \ref{maincase},
$I_{A(B_1)}$ has a Gr\"obner basis consisting of
the following binomials:
\begin{enumerate}
\item[(a)]
$
x_{1 M_1}
x_{2 M_2}
x_{3 M_3}
-
x_{4 M_1}
x_{5 M_2}
x_{6 M_3}
$,
\item[(b)]
$
x_{j_1 j_2 M_1}
x_{j_3 M_2}
-
x_{j_4 j_5 M_1}
x_{j_6 M_2}
$,
where $\{j_1,j_2,j_3\}=\{1,2,3\}$ and $\{j_4,j_5,j_6\}=\{4,5,6\}$,
\item[(c)]
$x_{j_1\cdots j_n} x_{j_{n+1} \cdots j_{2n}}
-x_{j_1' j_3' \cdots j_{2n-1}'} x_{j_{2}' j_4' \cdots j_{2n}'}$,
where ${\rm sort}(j_1 \cdots j_{2n})= j_1' \cdots j_{2n}'$.
\end{enumerate}
}
\end{Example}

Since the Gr\"obner basis in Example \ref{bexample} is not
quadratic, we have to consider another monomial order
to find a quadratic Gr\"obner basis.

\begin{Remark}
{\em
In \cite{Haase}, they say that L. Piechnik and C. Haase proved 
that
the toric ideal of the multiple $2 n {\mathcal B}_3$ possesses a squarefree quadratic
initial ideal for $n > 1$.
This fact is directly obtained by Theorem \ref{maincase} since
the toric ideal of the multiple $2  {\mathcal B}_3$ possesses a squarefree quadratic
initial ideal.
Similarly,
since
the toric ideal of the multiple $3  {\mathcal B}_3$ possesses a squarefree quadratic
initial ideal, Theorem \ref{maincase} guarantees that 
the toric ideal of the multiple $3 n {\mathcal B}_3$ possesses a squarefree quadratic
initial ideal for $n > 1$.
However, since there are infinitely many prime numbers,
it is difficult to show the existence of a squarefree quadratic
initial ideal of the toric ideal of 
$m {\mathcal B}_3$ for all $m >1$ in this way.
}
\end{Remark}

\begin{Theorem}
\label{Birkhoff}
Let
$A = \{t_1^n \}$ with $n >1$ and suppose that $B_1$
satisfies $\sharp|B_1|=6$ and $I_{B_1}= \left<z_1z_2z_3-z_4z_5z_6 \right>$.
Then, $I_{A(B_1)}$ has a quadratic Gr\"obner basis 
consisting of the following binomials:
\begin{enumerate}
\item[(i)]
$
\underline{
x_{j_1 j_2 M_1}
x_{j_3 M_2}
}
-
x_{j_4 j_5 M_1}
x_{j_6 M_2}
$
where $\{j_1,j_2,j_3\}=\{1,2,3\}$ and $\{j_4,j_5,j_6\}=\{4,5,6\}$,
\item[(ii)]
$
\underline{
x_{j_1\cdots j_n} x_{j_{n+1} \cdots j_{2n}}
}
-x_{1 \cdots 1 j_1' \cdots j_\alpha'} x_{1 \cdots 1 j_{\alpha + 1}' \cdots j_{2\alpha}'}$
where
${\rm sort} (j_1 \cdots j_{2n})= 1 \cdots 1 j_1' \cdots j_{2 \alpha}'$ and
$j_2' >1$.
\end{enumerate}
\end{Theorem}

\begin{proof}
Let ${\mathcal G}$ denote the set of binomials above.
Since $A = \{t_1^n\}$, each binomial in (ii) and (iii) belongs to
$I_{A(B_1)}$.
In addition, thanks to Lemma \ref{alemma},
each binomial in (i) belongs to $I_{A(B_1)}$.
Hence ${\mathcal G}$ is a (finite) subset of $I_{A(B_1)}$.

\bigskip

\noindent
{\bf Claim 1.}
There exists a monomial order such that the initial monomial of 
each binomial in ${\mathcal G}$ is the underlined monomial.

By virtue of \cite[Theorem 3.12]{Stu},
it is enough to show that the reduction modulo ${\mathcal G}$ is Noetherian.
Suppose that there exists a sequence of reductions modulo ${\mathcal G}$ which
does not terminate.
Let $v$ be a monomial in $K[\xb]$ and assume
$v \stackrel{g}{\longrightarrow} v'$ with $g \in {\mathcal G}$. 
Then we have
$$
\left\{
\begin{array}{cl}
\varphi_1(v) >_1 \varphi_1(v') & \ \ \mbox{if } g \mbox{ in (i)},\\
\varphi_1(v) = \varphi_1(v') & \ \ \mbox{if } g \mbox{ in (ii)}.
\end{array}
\right.
$$
Hence the number of binomials in (i) appearing in the sequence is finite.
Thus we may assume that the binomials in (i) do not appear in
the sequence.
Let
$
v=\prod_{\ell=1}^q
x_{i_{(\ell-1)r+1}  \cdots i_{\ell r}}$,
$
v'=\prod_{\ell=1}^q
x_{i_{(\ell-1)r+1}' \cdots i_{\ell r}' }$
and let
$
m_\ell
$
(resp. 
$
m_\ell'
$
)
denote the number of $1$'s appearing in
$
i_{(\ell-1)r+1}  \cdots i_{\ell r}
$
(resp. 
$
i_{(\ell-1)r+1}'  \cdots i_{\ell r}'
$
).
Then,
we have
$$
\sum_{1 \leq \ell_1 < \ell_2 \leq q} \left| m_{\ell_1} - m_{\ell_2} \right|
\ \ \ 
\geq
\ \ 
\sum_{1 \leq \ell_1 < \ell_2 \leq q} \left| m_{\ell_1}' - m_{\ell_2}' \right|
$$
if
$g \in {\mathcal G}$ belongs to (ii).
(The equality holds if and only if 
$g=
\underline{
x_{j_1\cdots j_n} x_{j_{n+1} \cdots j_{2n}}
}
-x_{1 \cdots 1 j_1' \cdots j_\alpha'} x_{1 \cdots 1 j_{\alpha + 1}' \cdots j_{2\alpha}'}$
satisfies that
the difference between
the number of 1's in $j_1\cdots j_n$ 
and
that in
$j_{n+1} \cdots j_{2n}$ is
at most one.)
Hence, 
we may assume that $1$'s in the indices is stable.
Then, since the inversion number is strictly decreasing
in the sequence of reductions modulo binomials in (ii),
the sequence is finite.

\bigskip

\noindent
{\bf Claim 2.}
The set ${\mathcal G}$ is a Gr\"obner basis of $I_{A(B_1)}$.

Suppose that ${\mathcal G}$ is not a Gr\"obner basis of $I_{A(B_1)}$.
Then there exists a binomial $0 \neq g=p_1 -p_2 \in I_{A(B_1)}$ such that
neither $p_1$ nor $p_2$ is divisible by the initial monomial of 
any binomial in ${\mathcal G}$.
Let
$p_1 = 
\prod_{\ell=1}^p
x_{i_{(\ell-1)r+1}  \cdots i_{\ell r}}, \ 
p_2=
\prod_{\ell=1}^p
x_{i_{(\ell-1)r+1}' \cdots i_{\ell r}' }$.
By Lemma \ref{keylemma},
we have $\varphi_1(p_1) -\varphi_1 (p_2)
= \prod_{\xi =1}^{pr} z_{i_\xi} - \prod_{\xi =1}^{pr} z_{i_\xi'}
 \in \left<z_1z_2z_3-z_4z_5z_6 \right>$.

Suppose that $\prod_{\xi =1}^{pr} z_{i_\xi} - \prod_{\xi =1}^{pr} z_{i_\xi'} \neq 0$.
We may assume that $\prod_{\xi =1}^{pr} z_{i_\xi}$ is divided by $z_1z_2z_3$.
Since $p_1$ is not divided by the initial monomial of any binomial in (i),
$p_1$ is divided by a cubic monomial $ x_{ 1 M_1 }x_{  2  M_2 }x_{  3  M_3 }$
where $2,3 \notin M_1$, $1,3 \notin M_2$ and $1,2 \notin M_3$.
Note that $M_i \neq \emptyset$ by $n >1$.
Since $p_1$ is not divided by the initial monomial of any binomial in (ii),
the number of $1$'s in $i M_i$ is differ by at most one.
Since $1$ appears in neither $ 2  M_2$ nor $3  M_3$, we have $1 \notin M_1$.
Thus $M_1 \subset  \{4,5,6\}$.
Then $p_1$ is divided by the initial monomial of the binomial
$g=x_{ 1 M_1} x_{ 2  M_2} -x_{ 1  2 M_1'}x_{ M_2'}$
where
${\rm sort} (1 M_1 2  M_2) = 1 2 M_1' M_2'$
and $g$ belongs to (ii).

Suppose that $\prod_{\xi =1}^{pr} z_{i_\xi} - \prod_{\xi =1}^{pr} z_{i_\xi'} = 0$.
Since neither $p_1$ nor $p_2$ is divisible by the initial monomial of any binomial in (ii),
there exists $0 \leq p' \leq p$ and $0 \leq \beta \leq r$ such that
$$
p_1 = p_2 = 
\prod_{\ell=1}^{p'}
x_{\zeta_{(\ell-1)r+1}  \cdots \zeta_{\ell r}}
\prod_{\ell=p'+1}^p
x_{\theta_{(\ell-1)r+1}  \cdots \theta_{\ell r}}
$$
where
$\zeta_{(\ell-1)r+ \eta }  = 1$ for all $1 \leq \eta \leq \beta$,
$\theta_{(\ell-1)r+ \eta }  = 1$ for all $1 \leq \eta \leq \beta-1$
and
$
\zeta_{\beta +1} \leq \cdots \leq \zeta_{r}
\leq
\zeta_{r+ \beta +1} \leq \cdots \leq \zeta_{2 r}
\leq \cdots \leq
\zeta_{ (p'-1)r + \beta +1} \leq \cdots \leq \zeta_{p' r}
\leq
\theta_{p' r+\beta} \leq \cdots \leq \theta_{(p'+1)r}
\leq
\theta_{(p'+1)r+ \beta} \leq \cdots \leq \theta_{(p'+2) r}
\leq \cdots \leq
\theta_{(p-1)r+ \beta } \leq \cdots \leq \theta_{p r}
$.
Hence
$g = p_1 - p_2 = 0$ and this is a contradiction.

Thus,
there exists no binomial $0 \neq g=p_1 -p_2 \in I_{A(B_1)}$ such that
neither $p_1$ nor $p_2$ is divisible by the initial monomial of 
any binomial in ${\mathcal G}$
and hence
${\mathcal G}$ is a Gr\"obner basis of $I_{A(B_1)}$
as desired.
\end{proof}

\section{Observation}
Finally, we conclude this paper with 
a summary of our algebraic theory of nested configurations.
For a configuration $A$, let ${\mathcal G}_<$ denote the reduced Gr\"obner basis of $I_A$ 
with respect to a monomial order $<$.
Let 
$$\lambda(A)
:=
\min_{<} \left( \max\left( \deg(g) \ | \ g \in {\mathcal G}_< \right) \right)
.$$
(If $I_A=(0)$, then we set $\lambda(A)=0$.)
Thanks to the results in Section 2, if $\lambda(A(B_1,\ldots,B_d)) \neq 0$,
then
$$  
\max(2, \lambda(A) ) \leq \lambda(A(B_1,\ldots,B_d)) \leq \max(2,\lambda(A),\lambda(B_1),\ldots ,\lambda(B_d))
.$$
Moreover, if $\lambda(A(B_1,\ldots,B_d)) \neq 0$ and $A$ consists of squarefree monomials
then
$$  
\lambda(A(B_1,\ldots,B_d)) = \max(2,\lambda(A),\lambda(B_1),\ldots ,\lambda(B_d))
.$$
Let $n \geq 2$ be an integer and let $X$ be the one of the following algebraic properties:
\begin{enumerate}
\item
The toric ring is normal;
\item
The toric ideal has a squarefree initial ideal;
\item
The toric ideal has a quadratic initial ideal;
\item
The toric ideal has a squarefree quadratic initial ideal;
\item
The toric ideal has an initial ideal of degree $\leq n$;
\item
The toric ideal is generated by quadratic binomials;
\item
The toric ideal is generated by binomials of degree $\leq n$.
\end{enumerate}
Then we have
\begin{eqnarray*}
A, B_1,\ldots,B_d \mbox{ have the property } X &\Longrightarrow  &
A(B_1,\ldots,B_d)   \mbox{ has the property } X.\\
A(B_1,\ldots,B_d)   \mbox{ has the property } X
&\Longrightarrow  &
A \mbox{ has the property } X .
\end{eqnarray*}
Moreover, if $A$ consists of squarefree monomials, then we have
$$
A, B_1,\ldots,B_d \mbox{ have the property } X \Longleftrightarrow 
A(B_1,\ldots,B_d)   \mbox{ has the property } X.
$$

\bigskip

\noindent
{\bf Acknowledgement.}
This research was supported by JST, CREST.

\bigskip

\noindent
Hidefumi Ohsugi\\
Department of Mathematics\\
College of Science\\
Rikkyo University\\
Toshima-ku, Tokyo 171-8501, Japan\\
{\tt ohsugi@rikkyo.ac.jp}

\bigskip

\noindent
Takayuki Hibi\\
Department of Pure and Applied Mathematics\\
Graduate School of Information Science and Technology\\
Osaka University\\
Toyonaka, Osaka 560-0043, Japan\\
{\tt hibi@math.sci.osaka-u.ac.jp}

\end{document}